\documentclass[a4paper,twoside,12pt,reqno]{amsart}
\usepackage[utf8]{inputenc}
\usepackage[T1]{fontenc}
\usepackage[english]{babel}
\usepackage{tgpagella}
\usepackage{hyperref}
\usepackage{amscd} 
\usepackage{csquotes}
\usepackage{fancyhdr}
\usepackage{multirow}
\usepackage{cleveref}
\usepackage{comment}

\usepackage{amsthm, amsfonts, amssymb, amsmath, latexsym, enumerate,array,color}
\usepackage{amscd} 
\usepackage[all]{xy}
\usepackage{pstricks,graphicx}
\usepackage{stmaryrd}
\usepackage{hyperref,mdwlist,mathrsfs}
\usepackage{tikz}
\usepackage{pgf,tikz}
\usetikzlibrary{arrows}

\usepackage{color}

\usepackage{amssymb}
\usepackage{amsmath}
\usepackage{amsthm}
\usepackage{paralist}

\newtheorem{theorem}{Theorem}[section]
\newtheorem{lemma}[theorem]{Lemma}
\newtheorem{question}[theorem]{Question}
\newtheorem{proposition}[theorem]{Proposition}
\newtheorem{corollary}[theorem]{Corollary}
 
\newtheorem*{proposition*}{Proposition}
\theoremstyle{definition}
\newtheorem{definition}[theorem]{Definition}
\newtheorem{example}[theorem]{Example}
\newtheorem{remark}[theorem]{Remark}
\newtheorem*{notation}{Notation}

\DeclareMathOperator{\Ann}{Ann}
\DeclareMathOperator{\coker}{coker}
\DeclareMathOperator{\Ext}{Ext}

\DeclareMathOperator{\Hom}{Hom}
\DeclareMathOperator{\im}{im}
\DeclareMathOperator{\Tor}{Tor}
\DeclareMathOperator{\Soc}{Soc}
\DeclareMathOperator{\reg}{reg}

\newcommand{\fa}{\mathfrak a}
\newcommand{\fb}{\mathfrak b}
\newcommand{\fm}{\mathfrak m}

\newcommand{\x}{\mathbf{x}}
\newcommand{\y}{\mathbf{y}}

\newcommand{\tA}{\tilde{A}}
\newcommand{\tfa}{\tilde{\fa}}
\newcommand{\tB}{\tilde{B}}
\newcommand{\tfb}{\tilde{\fb}}
\newcommand{\tC}{\tilde{C}}

\numberwithin{equation}{section}

\newcommand{\AS}[1]{
  {\color{magenta} Alexandra: #1}}
\newcommand{\na}[1]{
  {\color{cyan} #1}}

\title[Betti numbers for connected sums of graded AGA's]{Betti numbers for connected sums of graded Gorenstein Artinian algebras}

\author[]{Nasrin Altafi}
\address{Nasrin Altafi: Department of Mathematics, KTH Royal Institute of Technology, S-100 44 Stockholm, Sweden and Department of Mathematics, Queen's University, 505 Jeffery Hall, University Avenue, Queen's University, Kingston, Ontario, Canada K7L 3N6}
\email{nar3@queensu.ca}

\author[]{Roberta Di Gennaro}
\address{Roberta Di Gennaro: Dipartimento di Matematica e Applicazioni \lq\lq Renato Caccioppoli\rq\rq, Complesso Universitario Monte Sant'Angelo, 
 Universit\`{a} degli Studi di Napoli Federico II, Via Cinthia 
  80126 Napoli, Italy}
\email{digennar@unina.it}

\author[]{Federico Galetto}
\address{Federico Galetto: Department of Mathematics and Statistics, Cleveland State University, 2121 Euclid Avenue, RT 1515
Cleveland, OH 44115-2215, USA }
\email{f.galetto@csuohio.edu}

\author[]{Sean Grate}
\address{Sean Grate: Department of Mathematics and Statistics, Auburn University, 221 Parker Hall, Auburn, AL 36849, USA}
\email{sean.grate@auburn.edu}

  \author[]{Rosa M.\ Mir\'o-Roig}
  \address{Rosa Maria Mir\'o-Roig: Facultat de
  Matem\`atiques i Inform\`atica, Universitat de Barcelona, Gran Via des les
  Corts Catalanes 585, 08007 Barcelona, Spain} \email{miro@ub.edu,  ORCID 0000-0003-1375-6547}

 \author[]{Uwe Nagel} 
 \address{Uwe Nagel: Department of
  Mathematics, University of Kentucky, 715 Patterson Office Tower,
  Lexington, KY 40506-0027, USA}
  \email{uwe.nagel@uky.edu}
  
\author[]{Alexandra Seceleanu}
\address{Alexandra Seceleanu: Department of
  Mathematics, University of Nebraska-Lincoln, 203 Avery Hall, Lincoln, NE 68588, USA}
\email{aseceleanu@unl.edu}

\author[]{Junzo Watanabe}
\address{Department of Mathematics Tokai University, Hiratsuka, Kanagawa 259--1292, Japan}
\email{watanabe.junzo@tokai-u.jp}

\thanks{\hspace{-15pt}  Altafi was supported by Swedish Research Council grant VR2021-00472, Galetto was supported by NSF DMS--2200844, Mir\'o-Roig was partially supported by the grant PID2020-113674GB-I00, Nagel was partially supported by Simons Foundation grant \#636513,   Seceleanu was supported by NSF DMS--2101225.
}

\theoremstyle{definition}

\begin{document}
\maketitle
\begin{abstract} The connected sum construction, which takes as input Gorenstein rings and produces new Gorenstein rings, can be considered as an algebraic analogue for the topological construction having the same name. We determine the graded Betti numbers for connected sums of graded Artinian Gorenstein algebras. Along the way, we find the graded Betti numbers for fiber products of graded rings; an analogous result was obtained in the local case by Geller \cite{G}.  We relate the connected sum construction to the doubling construction, which also produces Gorenstein rings. Specifically, we show that a connected sum of doublings is the doubling of a fiber product ring. 

\end{abstract}

\section{Introduction} 

The connected sum is a topological construction that takes two manifolds to produce a new manifold \cite[p. 7]{Massey}. An algebraic analog of this surgery construction was introduced by H. Ananthnarayan, L. Avra\-mov, and W.F. Moore in their paper \cite{AAM} in the local case. In this paper, we elucidate some properties of this construction in the graded case. 

Let $A$ and $B$ be two graded Artinian Gorenstein (AG) $K$-algebras with the same socle degree $d$, let $T$ be an AG $K$-algebra of socle degree $k<d$, and suppose there are surjective maps $\pi_A\colon A\rightarrow T$, and $\pi_B\colon B\rightarrow T$.  From this data, one forms the fiber product algebra $A\times_TB$ as the categorical pullback of $\pi_A,\pi_B$; the connected sum algebra $A\#_TB$ is the quotient of $A\times_TB$ by a certain principal ideal $\langle (\tau_A,\tau_B)\rangle\subset A\times_TB$.  The connected sum is again an AG $K$-algebra (see Definition \ref{Def_CS}). 
As mentioned, this algebraic connected sum operation for local Gorenstein algebras $A,B$ over a local Cohen-Macaulay algebra $T$ was introduced in \cite{AAM}.

In \cite{G} and  \cite{CGS}, the authors determined the minimal free resolution of a two-factor fiber product $A\times_TB$ of local rings. 
In this paper, we extend  their work to the setting of fiber products of graded rings and generalize it to fiber products involving multiple factors. We also consider connected sums with multiple sumands and we answer the following question:
\begin{question}\label{Q1} Fix $A_1, \ldots, A_r$ graded AG $K$-algebras with the same socle degree.
    What are the graded Betti numbers of their fiber product over $K$?
What are the graded Betti numbers of their connected sum over $K$?
 \end{question}

 Our first series of main results answers the above question. For specific formulas we refer the reader to Theorem \ref{thm:Betti fiber product}, Theorem \ref{thm:Betti conn sum}, Theorem \ref{thm:Betti multifactor conn sum}, Theorem \ref{thm:Betti multifactor conn sum} and their corollaries.

 Celikbas, Laxmi and Weyman solved a particular case of Question \ref{Q1} in \cite[Corollary 6.3]{CLW}. Specifically, they determined  a minimal free resolution of the connected sum of $K$-algebras $A_i
:= K[x_i]/(x_i^{d_i})$ by using the doubling construction (see section \ref{s:doubling}). A second goal of this paper is to  generalize their result and investigate conditions for a connected sum of  AG $K$-algebras $A_1, \dots , A_r$ with the same socle degree to be a doubling. More precisely we ask:

\begin{question} \label{Q2}
 Assume that $A_1,\ldots,A_r$ are graded AG $K$-algebras with the same socle degree. Is the connected sum $A=A_1 \# _K \cdots \# _KA_r$ a doubling? More precisely: if $A_i$ is a doubling of $\Tilde{A_i}$, is $A$ a doubling of $\Tilde{A_1} \times_K\cdots \times_K \Tilde{A_r}$?
\end{question}
We answer the above question in the affirmative in Theorem \ref{thm:doubling}.

Our paper is structured as follows: section \ref{s: background} introduces the necessary background and develops the basic properties of multi-factor fiber products and connected sums, section \ref{sec:Betti numbers} computes the graded Betti numbers for multi-factor fiber products and connected sums, and section \ref{sec:doubling}  analyzes connected sums that arise as doublings of certain fiber products.

\noindent \textbf{Acknowledgement.} The project got started at the  meeting ``Workshop on Lefschetz Properties in Algebra, Geometry, Topology and Combinatorics'', held at the Fields Institute in Toronto, Canada, May 15--19, 2023. The authors would like to thank the Fields Institute and the organizers for the invitation and financial support. Additionally, we thank Graham Denham for asking a question which motivated our work, and Mats Boij for useful discussions.

\section{background}
\label{s: background}
In this section, we fix some notation and recall some basic facts on Artinian Gorenstein (AG) algebras, fiber products, connected sums of graded Artinian algebras, as well as on Macaulay dual generators needed in the sequel.
\smallskip

\subsection{Oriented AG algebras} Throughout this paper, $K$ is an arbitrary field. Given a graded $K$-algebra $A$, its homogeneous maximal ideal is $m_A =\oplus _{i\ge 1} A_i$. A $K$-algebra $A$ is called {\em Artinian} if it is a finite dimensional vector space over $K$. The {\em socle} of an Artinian $K$-algebra $A$ is the ideal $(0 : m_A)$; its {\em socle degree} is the largest integer $d$ such that $A_d\ne 0$. The socle degree of an Artinian $K$-algebra agrees with its Castelnuovo-Mumford regularity, which is denoted by $\reg(A)$. The {\em type} of $A$ is the vector space dimension of its socle. 

The {\em Hilbert series} of a graded  $K$-algebra $ A$ is the generating function $H_A( t) =\sum _{i\ge 0} \dim(A_i)t^i$.
 The {\em Hilbert function} $HF_A$ of a $K$-algebra $A$ is the sequence of coefficients of its
Hilbert series. 

Suppose that $A$ has a presentation $A=R/I$ as a quotient of a graded $K$-algebra $R$. 
The graded Betti numbers of $A$ over $R$ are the integers $\beta_{ij}^R(A)=\dim_K \Tor^R_i(A,K)_j$. These homological invariants are our main focus. The graded Poincar\'{e} series of $A$ over $R$ is the generating function $P^R_A(t,s)=\sum_{i,j} \beta_{ij}^R(A) t^is^j$. If $R$ is regular, then the Poincar\'{e} series is in fact a polynomial.

A graded Artinian $K$-algebra $A$  with socle degree $d$ is said to be  {\em Gorenstein} if its
socle $(0:m_A)$ is a one dimensional $K$-vector space. For any Artinian Gorenstein   graded $K$-algebra  $A$ with socle degree $d$  and for any
non-zero morphism of graded vector spaces $f_A:A\to K(-d)$, known as  an orientation of $A$,
there is a pairing    
\begin{equation}\label{pairing}
A_{i}\times A_{d-i}\to K \text{ defined by } (a_i,a_{d-i})\mapsto f_A(a_ia_{d-i})
\end{equation}
 which 
is non-degenerate. We call the pair $(A,f_A)$ an 
{\em oriented AG $K$-algebra}.

\begin{definition}[{\cite[Lemma 2.1]{IMS}}] Let $(A,f_A)$ and $(T,f_T)$
be two oriented AG $K$-algebras with $\reg(A)=d$ and $\reg(T)=k$, and let $\pi: A \to T$ be a graded map. There exists a unique homogeneous element $\tau_A \in A_{d-k}$  such that $f_A(\tau a)=f_T(\pi (a))$ for all $a\in A$; we call it the {\em Thom class} for $\pi : A \to  T$. 
\end{definition}

\begin{remark}
\label{rem: Thom class}
Restating \cite[Remark 2.8]{IMS}, we have that $\tau_A$ is the image of $1\in T$ under the composite map $T(-k)\cong \Ext^n(T,Q)\to \Ext^n(A,Q)\cong A(-d)$, where the middle map is $\Ext^n(\pi,Q)$.
\end{remark}

\begin{example}
 Let $(A,f_A)$ be an oriented AG $K$-algebra with socle degree $\reg(A)=d$. Consider $(K,f_K)$ where $f_K:K\to K$ is the identity map. Then the Thom class for the canonical projection $\pi:A\to K$ is the unique element $s\in A_d$ such that $f_A(s)=1$.
\end{example}

Note
that the  Thom class for $\pi : A \to  T$ depends not only on the map $\pi $, but also on the orientations chosen for $A$ and $T$.

\subsection{Macaulay dual generators} 
\label{sect:MDG}
Let $Q=K[x_1,\ldots,x_n]$ be a polynomial ring and let $Q'=K[X_1,\ldots,X_n]$ be a divided power algebra, regarded as a $Q$-module with the contraction action 
\[
x_i\circ X_j^k=\begin{cases}X_j^{k-1}\delta_{ij} & \text{if} \ k>0\\ 0 & \text{otherwise}\\ \end{cases}
\]
where $\delta_{ij}$ is the Kronecker delta. We regard $Q$ as a graded $K$-algebra with $\deg X_i=\deg x_i$.

For each degree $i\geq 0$, the action of $Q$ on $Q'$ defines a non-degenerate $K$-bilinear pairing 
\begin{equation}
\label{eq:MDPairing}
    Q_i \times Q'_i \longrightarrow K \text{ with } (f,F) \longmapsto f \circ F.
\end{equation}
This implies that for each $i\geq 0$ we have an isomorphism of $K$-vector spaces $Q'_i\cong \Hom_K(Q_i,K)$ given by $F\mapsto\left\{f\mapsto f\circ F\right\}$.

 It is a classical result of Macaulay \cite{Macaulay} (cf. \cite[Lemma 2.14]{IK}) that an Artinian $K$-algebra $A=Q/I$ is Gorenstein with socle degree $d$ if and only if  $I=\Ann_Q(F)=\{f\in Q\mid f\circ F=0\}$ for some homogeneous polynomial $F\in Q'_d$.  Moreover, this polynomial, termed a {\em Macaulay dual generator} for $A$, is unique up to a scalar multiple. 

A choice of orientation on $A$ corresponds to a choice of Macaulay dual generator. Every orientation on $A$ can be written as the function $f_A:A\to K$ defined by $f_A(g)\mapsto (g\circ F)(0)$ for some Macaulay dual generator $F$ of $A$ (the notation $(g\circ F)(0)$ refers to evaluating the element $g\circ F$ of $Q'$ at $X_i=0$).

\subsection{Fiber product} 
We start by recalling the definition of the fiber product.

\begin{definition} Let $A$, $B$ and $T$ be graded $K$-algebras  and   $\pi _A : A \to 
T$ and $\pi _B : B \to  T$  morphisms of graded $K$-algebras. We define the {\em fiber product} of $A$ and $B$ over $T$ as the
graded $K$-subalgebra of $A\oplus B$:
$$
A \times _T B = \{(a, b) \in A\oplus B \mid  \pi_A(a) = \pi _B(b) \} .
$$
\end{definition}

If $\pi _A$ and $\pi _B$ are surjective, then there is a degree-preserving exact sequence
\begin{equation}\label{exactFP}
    0\to A\times _T B\to A\oplus B \to T \to 0
\end{equation}
which allows to compute the Hilbert series of the fiber product as
\begin{equation}\label{HilbertFP}
 HF_{A \times _T B}( t) = HF_A( t) + HF_B( t)  -  HF_T( t).
\end{equation}

While presentations of arbitrary fiber products can be unruly, the case $T=K$ is best-behaved.

\begin{lemma}\label{lem:fiber product over K}
Let $R = K[x_1,\ldots,x_m]$ and $S = K[y_1,\ldots,y_n]$ be polynomial rings over  $K$ with homogeneous maximal ideals $\x = (x_1,\ldots,x_m)$ and $\y = (y_1,\ldots,y_n)$, respectively. Let  $Q = R \otimes_K S = K[x_1,\ldots,x_m, y_1,\ldots,y_n]$. If $A = R/\fa$ and $B = S/\fb$ have canonical projections $\pi_A:A\to K$ and $\pi_B:B\to K$, then the fiber product over $K$ has presentation
\begin{equation}\label{eq:fiber product presentation}
A\times_K B=\frac{Q}{\x \cap \y  +\fa \ +\fb },
\end{equation}
where in \eqref{eq:fiber product presentation} $\fa, \fb, \x,\y$ denote extensions of the respective ideals to $Q$. 
In particular, if $A$ and $B$ are graded, then $A\times_K B$ is a bigraded algebra with 
\[
[A\times_K B]_{(i,j)}=\begin{cases}
K & \text{ if } (i,j)= (0,0), \\
A_i\oplus B_j & \text{ if } (i,j)\neq (0,0).
\end{cases}
\]
\end{lemma}
\begin{proof}
The presentation of the fiber product is given in \cite[Proposition 3.12]{IMS}. The fact that the fiber product is bigraded follows from noticing that the relations in \eqref{eq:fiber product presentation} are homogeneous with respect to natural  bigrading of $Q$. Finally, the formula for the graded components of  $A\times_K B$ follows from \eqref{exactFP}, which can be interpreted as an exact sequence of bigraded vector spaces.
\end{proof}

\begin{example}\label{ex1} Consider the standard graded complete intersection algebras 
\[
A = \frac{K[x,y,z]}{(x^3,y^4,z^4)} \ \text{ and } B = \frac{K[u,v]}{(u^5,v^5)}.
\]
 Their Hilbert functions are given by
\begin{eqnarray*}
HF_A &=& (1,3,6,9, 10, 9,6,3,1) \ \text{ and }\\
HF_B &=& (1,2,3,4,5 , 4,3,2,1). 
\end{eqnarray*}
Set $R=K[x,y,z]$ and $S=K[u,v]$. The minimal free resolutions of $A$ and $B$ are the Koszul complexes
$$
0 \to R(-11) \to R(-7)^2\oplus R(-8)\to R(-4)^2\oplus R(-3)\to R \to A \to 0,
$$
and
$$ 0\to S(-10)\to S(-5)^2\to S \to B \to 0.
$$
The fiber product $C=A\times _K B$ of $A$ and $B$ and its Hilbert function are
$$C= K[x,y,z,u,v]/(xu,xv,yu,yv,zu,zv,x^3,y^4,z^4,u^5,v^5),$$
and
$$HF_C=(1,5, 9, 13, 15, 13, 9, 5,2).$$
The Betti table of $C$ as a $K[x,y,z,u,v]$-module is shown in \Cref{tab: betti1}.
\begin{table}[ht]
    \centering
$$    \begin{tabular}{c|cccccc}
        & 0 & 1 & 2&3&4&5\\
       \hline
       total & 1 & 11 & 25&24&11&2\\
       \hline
       0: & 1 & . & . & . & . & .\\
      1: & . & 6 & 9 & 5 & 1 & .\\
      2: & . & 1 & 2 & 1 & . & .\\
      3: & . & 2 & 4 & 2 & . & .\\
      4: & . & 2 & 6 & 6 & 2 & .\\
      5: & . & . & 2 & 4 & 2 & .\\
      6: & . & . & 1 & 2 & 1 & .\\
      7: & . & . & . & . & . & .\\
      8: & . & . & 1 & 4 & 5 & 2
    \end{tabular}
  $$
  \caption{Betti table of $C$ in \Cref{ex1}}
    \label{tab: betti1}
\end{table}

  Note that $C$ is an Artinian level $K$-algebra of type 2, i.e., all the elements of its socle have the same degree and the socle has dimension 2.
\end{example}

Recall that an Artinian $K$-algebra $A$ has the {\em strong Lefschetz property} (SLP) if there exists a linear form $\ell$ such that the multiplication map
$\times \ell^k : A_i \rightarrow A_{i+k}$
has maximal rank (i.e, it is injective or surjective) for all $i$ and $k$. It is known that if $A$ and $B$ are two AG $K$-algebras with the same socle degree, and both have the SLP, then $A \times_K B$ also has the SLP \cite[Proposition 5.6]{IMS}.

We also consider multi-factor fiber products, which we now define.

\begin{definition} 
\label{def: multi factor FP}
Let $A_1,\ldots, A_r$ and $T$ be graded $K$-algebras  and let  $\pi _i : A_i \to 
T$  morphisms of graded $K$-algebras. We define the  {\em fiber product} of $A_1,\dots,A_r$ over $T$ as
\begin{multline}
    A_1\times_T  \cdots \times_T A_r=\\
    \{(a_1, \ldots, a_r) \in A_1\oplus \cdots \oplus A_r \mid  \pi_i(a_i) = \pi_j(a_j), 1\leq i,j\leq r\} .
\end{multline}
\end{definition}

\begin{remark}\label{rem:multi-factor FP as iterated FP}
The multi-factor fiber product construction coincides with i\-te\-ra\-tively applying the two-factor fiber product construction to the list $A_1,\ldots, A_r$ a total of $r-1$ times, that is
\[
A_1\times_T  \cdots \times_T A_r=\left((A_1\times_T A_2)\times_T \cdots \right )\times_T A_r.
\]
\end{remark}

 We will need the following generalizations of equations \eqref{exactFP} and  \eqref{eq:fiber product presentation}, which describe a presentation for fiber products with arbitrary many summands over the residue field. 

\begin{lemma} \label{lem: multifactor FP presentation}
Let $R_1,\dots, R_r$ be polynomial rings over $K$ with maximal ideals $\mathbf{x}_1,\dots,\mathbf{x}_r$,
and let $Q=R_1\otimes_K \dots \otimes_K R_r$.
Suppose $A_i = R_i / \fa_i$, for some homogeneous ideal $\fa_i$ of $R_i$.
For $r\geq 2$, $A_1 \times_K \cdots \times_K A_r \cong Q/J$
with 
\[
J = \fa_1 +\cdots+\fa_r +  \sum_{1\leq i \neq j\leq r} (\x_i \cap \x_j )
\]
and there is an exact sequence of graded $Q$-modules
 \begin{equation}
 \label{exactFPr}
0 \to A_1\times_K\cdots \times_KA_r \to A_1\oplus\cdots \oplus A_r \to K^{r-1} \to 0. 
\end{equation}
\end{lemma} 

\begin{proof}
This follows by induction on $r$ with \eqref{eq:fiber product presentation} settling the case $r=2$. 

Setting $Q'=R_1\otimes_K \cdots \otimes_K R_{r-1}$, consider the ideal of $Q'$
\[
J'= \fa_1+\cdots+\fa_{r-1}+  \sum_{1\leq i < j\leq r-1} (\x_i \cap \x_j ).
\]
Applying \eqref{eq:fiber product presentation} to $A_1 \times_K \cdots \times_K A_{r-1} \cong Q'/J'$, we get
\begin{equation*}
A_1 \times_K \cdots \times_K A_r \cong Q'/J'  \times_K R_i/\fa_r \cong Q/J,
\end{equation*}
where the ideal $J$ is given by 
\begin{eqnarray*}
J &=& J'+\fa_r+\x_r\cap (\x_1+\cdots +\x_{r-1}) \\
 &=&  \fa_1+\cdots+\fa_{r-1} + \fa_r +  \sum_{1\leq i < j\leq r-1} (\x_i \cap \x_j ) + \sum_{1\leq i \leq r-1} (\x_i \cap \x_r )\\
 &=& \fa_1 +\cdots+\fa_r +  \sum_{1\leq i \neq j\leq r} (\x_i \cap \x_j ). 
\end{eqnarray*}
This yields the claimed presentation. 

By Definition \ref{def: multi factor FP}, we have that $A_1\times_K\cdots \times_KA_r$ is a $K$-subalgebra of $A_1\oplus\cdots \oplus A_r$. The inclusion $A_1\times_K\cdots \times_KA_r\subseteq A_1\oplus\cdots \oplus A_r$ gives the first nonzero map in \eqref{exactFPr}, while the second map can be defined by 
\[
(a_1,\ldots, a_r)\mapsto (\pi_2(a_2)-\pi_1(a_1), \pi_3(a_3)-\pi_1(a_1), \ldots, \pi_r(a_r)-\pi_1(a_1)),
\]
where the maps $\pi_i:A_i\to K$ are the canonical projections.
The claim regarding exactness of  \eqref{exactFPr} follows from \Cref{def: multi factor FP}.
\end{proof}

\subsection{Connected sum} \label{setup CS}
Let $(A,f_A)$, $(B,f_B)$ and $(T,f_B)$
be oriented AG $K$-algebras with
 $\reg(A)=\reg(B)=d$ and $\reg(T)=k$ and let 
$\pi_A : A \to  T$ and $\pi_B : B \to T$ be surjective graded $K$-algebra morphisms with Thom classes $\tau_A\in A_{d-k}$ and $\tau_B\in B_{d-k}$, respectively.
We assume that $\pi_A(\tau_A) = \pi_B(\tau_B)$, so that $(\tau_A,\tau_B) \in A \times _T B$.

\begin{definition}\label{Def_CS}
    The {\em connected sum} of the oriented AG $K$-algebras $A$ and $B$ over $T$ is the quotient ring of the fiber product $A\times _T B$ by the principal
ideal generated by the pair of Thom classes $(\tau _A, \tau _B)$, i.e.
$$ A \#  _TB =
(A \times _T B)/
\langle (\tau _A, \tau _B) \rangle .$$
\end{definition}

Note that this definition depends on $\pi_A$, $ \pi_B$ and the orientations on $A$ and  $B$.

By \cite[Lemma 3.7]{IMS} the connected sum is  characterized by the following exact sequence of vector spaces:
\begin{equation}\label{exactCS}
    0 \to  T(k - d) \to  A \times _T B \to  A\# _T B \to  0.
\end{equation}

Therefore, the Hilbert series of the connected sum satisfies
\begin{equation}\label{HilbertCS}
    HF_{A\# _T B}( t) = HF_A( t) + HF_B( t)- (1 + t^{d-k})HF_T( t).
\end{equation}

We recall the following characterization of connected sums.

\begin{theorem}[{\cite[Theorem 4.6]{IMS}}] \label{thm: dual generator CS}
Let $Q=K[x_1,\ldots,x_n]$ be a polynomial ring, and let $Q'=K[X_1,\ldots,X_n]$ be its dual ring (a divided power algebra).
	Let $F,G\in Q'_d$ be two linearly independent homogeneous forms of degree $d$, and suppose that there exists $\tau\in Q_{d-k}$ (for some $k<d$) satisfying
	\begin{enumerate}[(a)]
		\item $\tau\circ F=\tau\circ G\neq 0$, and 
		\item $\Ann(\tau\circ F=\tau\circ G)=\Ann(F)+\Ann(G)$.
	\end{enumerate} 
	Define the oriented AG $K$- algebras 
	$$A=\frac{Q}{\Ann(F)}, \ B=\frac{Q}{\Ann(G)}, \ T=\frac{Q}{\Ann(\tau\circ F=\tau\circ G)},$$
	and let $\pi_A\colon A\rightarrow T$ and $\pi_B\colon B\rightarrow T$ be the natural projection maps. Then there are $K$-algebra isomorphisms
	\begin{equation}\label{eq: conn sum Macaulay dual}
 A\times_TB\cong \frac{Q}{\Ann(F)\cap\Ann(G)}, \ \ A\#_TB\cong \frac{Q}{\Ann(F-G)}.
 \end{equation}
	Conversely, every connected sum $A\#_T B$ of graded AG $K$- algebras with the same socle degree over a graded AG $K$-algebra $T$ arises in this way.
\end{theorem}

In particular, when $T=K$ the polynomials $F$ and $G$ in the above theorem are polynomials expressed in disjoint sets of variables. The connected sum $A\#_K B$ is a graded $K$-algebra, but it is not bigraded.
Moreover, it is shown in \cite[Proposition 5.7]{IMS} that if $A$ and $B$ satisfy the SLP and they have the same socle degree, then $A\#_K B$ also satisfies the SLP.

\begin{example} \label{ex2} We will now build the connected sum of the  standard graded complete intersection $K$-algebras $A=K[x,y,z]/(x^3,y^4,z^4)$ and $B=K[u,v]/(u^5,v^5)$ described in Example \ref{ex1}. The connected sum $D = A\#_K B$ is isomorphic to
$$K[x,y,z,u,v]/(xu,xv,yu,yv,zu,zv,x^3,y^4,z^4,u^5,v^5,x^2y^3z^3+u^4v^4).$$
Its Hilbert function is
$$HF_D=(1,5, 9, 13, 15, 13, 9, 5,1) $$
and its Betti table is given in \Cref{tab: betti2}.
\begin{table}[ht]
    \centering
$$    \begin{tabular}{c|cccccc}
        & 0 & 1 & 2&3&4&5\\
       \hline
       total & 1 & 12 & 29&29&12&1\\
       \hline
       0: & 1 & .  & .& .&.&. \\
       1: & . & 6 & 9&5&1&. \\
       2:& . & 1 & 2&1&.&.\\
       3: & . & 2 & 4&2&.&. \\
       4: & . & 2 & 6 & 6 & 2 & .\\
       5: & .& .& 2&4 & 2&. \\
       6:&.&.&1&2&1&.\\
       7:&.&1&5&9&6&.\\
       8:&.&.&.&.&.&1
    \end{tabular}
$$    \caption{Betti table of $D$ in \Cref{ex2}}
    \label{tab: betti2}
\end{table}

So, $D$ is an AG $K$-algebra with socle degree 8.
\end{example}

An important feature of the connected sum of AG $K$-algebras is that it is also an AG $K$-algebra with the same socle degree as $A$ and $B$ (see \cite[Lemma 3.8]{IMS} or \cite[Theorem 1]{AAM}), in contrast to the fiber product which is an algebra of type two, hence not Gorenstein.

As before, we consider multi-factor connected sums. The multi-factor connected sum construction defined below coincides with ite\-ra\-tively applying the two-factor construction to the list $A_1,\ldots, A_r$ a total of $r-1$ times. In order to define this, we need to define an appropriate orientation and find the Thom class of a connected sum.

\begin{lemma}\label{lem: orientation CS}
Consider the setup of \S \ref{setup CS} and denote by $\tau_A$ and $\tau_B$ the Thom classes of $\pi_A$ and $\pi_B$ respectively. Then $A\#_TB$ is an oriented AG $K$-algebra with orientation $f:A\#_TB\to K$ defined by $f(a,b)=f_A(a)-f_B(b)$. 

Moreover, provided that $\pi_A(\tau_A)=0$, the surjective morphism
\[
\pi:A\#_T B\to T \text{ with } \pi(a,b)=\pi_A(a) =\pi_B(b)
\]
has Thom class $\tau=(\tau_A,0)$.
\end{lemma}
\begin{proof}
Recall from \Cref{thm: dual generator CS} that  if the Macaulay dual generators of $A$ and $B$ are $F$ and $G$ respectively (chosen to correspond to the given orientations $f_A$ and $f_B$), then the Macaulay dual generator of $A\#_T B$ is $F-G$. This defines an orientation by 
\[
g\mapsto \left(g\circ(F-G)\right)(0)=(g\circ F)(0) -(g\circ G)(0).
\]
If $g=(a,b)\in A\#_T B$, then $(g\circ F)(0)=f_A(a)$ and $(g\circ G)(0)=f_B(b)$.

To establish the claim regarding the Thom class we verify that
\[
f(\tau g)=f_T(\pi(g)) \text{ for all } g\in A\#_TB.
\]
With $g=(a,b)$, we have $\tau g=(\tau_A,0)(a,b)=(\tau_A a,0)$ since 
\[
f(\tau g)=f_A(\tau_Aa)=f_A(\pi_A(a))=f(\pi(g)).
\]
\end{proof}

We establish the convention that every connected sum in this paper will be oriented according to the orientation $f$ in Lemma \ref{lem: orientation CS}.

\begin{definition} 
\label{def: multi factor CS}
Let $A_1,\ldots, A_r$ and $T$ be graded AG $K$-algebras with socle degrees $\reg(A_i)=d$ and $\reg(T)=k$ and let  $\pi _i : A_i \to 
T$ be morphisms of graded $K$-algebras with Thom classes $\tau_1, \ldots, \tau_r$, respectively such that $\pi_i(\tau_i)=0$. We define the multi-factor {\em connected sum} $A_1\#_T\cdots \#_T A_r$ by
\begin{multline}
A_1\#_T\cdots \#_T A_r=\\
A_1\times_T \cdots \times_T A_r/\langle (\tau_1, 0,\ldots, 0,\tau_i, 0,\ldots, 0)\mid 2\leq i\leq r\rangle.    
\end{multline}
\end{definition}

\begin{remark}\label{rem:multi-factor CS as iterated CS}
Note that the above definition coincides with iterating the two-factor connected sum construction. Indeed, repeatedly applying Lemma \ref{lem: orientation CS} shows that the Thom class of the iterated connected sum
\[
((A_1\#_T A_2)\#_T A_3)\cdots \#_T A_{i-1}
\]
is $(\tau_1, 0,\ldots,0)$. Thus, to obtain 
\[
(((A_1\#_T A_2)\#_T A_3)\cdots \#_T A_{i-1})\#_T A_i,
\]
one goes modulo $\langle (\tau_1, 0,\ldots,0,\tau_i)\rangle$.
\end{remark}

\begin{lemma}
In the setup of Definition \ref{def: multi factor CS}, there is a short exact sequence
\begin{equation}\label{eq:ses multifactor CS}
0\to T(d-k)^{r-1}\xrightarrow{\alpha} A_1\times_T \cdots \times_T A_r \to A_1\#_T\cdots \#_T A_r \to 0,
\end{equation}
where the map $\alpha$ sends the generator of the $i$-th summand $T$ to the image of $(\tau_1, 0,\ldots,0, \tau_{i+1}, 0,\ldots, 0)$ in $A_1\times_T \cdots \times_T A_r$.
\end{lemma}
\begin{proof}
It is clear from Definition \ref{def: multi factor CS} that the connected sum is the cokernel of $\alpha$. It remains to justify that this map is injective. This follows from \cite[Lemma 2.6]{IMS}, where it is shown that the Gysin map $\iota:T(d-k)\to A_i$ that satisfies $\iota(1)=\tau_i$ is injective.
\end{proof}

%
%

\subsection{Doubling}
\label{s:doubling}

Let us start by recalling the doubling construction and some basic facts needed later on.

\begin{definition}
Set $R= K[x_1,\ldots,x_n]$. The \emph{canonical module} of a graded $R$-module $M$ is $\omega_M = \Ext^{n - \dim M}_R (M, R)$. 
\end{definition}

For example, one has $\omega_K \cong K$.

\begin{definition}\label{doubling} \cite[Section 2.5]{KKRSSY} Let $J\subset R$ be a homogeneous ideal of codimension $c$, such that $R/J$ is Cohen-Macaulay and $\omega_{R/J}$ is its canonical module. Furthermore, assume that $R/J$ satisfies the condition $G_0$ (i.e., it is Gorenstein at all minimal
primes). 
Let $I$ be an ideal of codimension $c + 1$. $I$ is called a {\em doubling } of $J$ via $\psi$ if there exists  a short exact sequence of $R/J$ modules
\begin{equation}\label{eq:doubling}
0 \rightarrow \omega_{R/J}(-d)\stackrel{\psi}{\rightarrow} R/J \rightarrow R/I\rightarrow 0.
\end{equation}
By \cite[Proposition 3.3.18]{BV}, if $I$ is a doubling, then $R/I$ is a Gorenstein ring. 
\end{definition}

Doubling plays an important role in the theory of Gorenstein liaison. Indeed, in \cite{KMMNP}, doubling is used to produce suitable Gorenstein divisors on arithmetically Cohen-Macaulay subschemes in several foundational constructions.
It is not true that every
Artinian Gorenstein ideal of codimension $c + 1$ is a doubling of some codimension $c$ ideal (see, for instance, 
\cite[Example 2.19]{KKRSSY}).

Moreover, the mapping cone of $\psi$ in \eqref{eq:doubling} gives a resolution of $R/I $. If it is minimal, then
one can read off the Betti table of $R/I$ from the Betti table of $R/J$. This mapping cone is the direct sum of the minimal free resolution $F_{\bullet }$ of $R/J$ with its dual (reversed) complex $\Hom(F_{\bullet},R)$ which justifies the terminology of "doubling".

\begin{lemma}
  \label{lem:facts}
  Let  $C = R/J$ be a Cohen-Macaulay $K$-algebra. Then: 
\begin{itemize}
\item[(a)]$\reg \omega_C = \dim C$; and

\item[(b)] if $C$ is a doubling of a Cohen-Macaulay $K$-algebra $\tilde{C}$, i.e., $\dim \tilde{C} = \dim C + 1$ and there is an exact sequence of graded $R$-modules
\[
0 \to \omega_{\tC} (-t) \to \tC \to C \to 0,
\]
then $ t = \reg C - \dim C$. 
\end{itemize}

\end{lemma} 

\begin{proof}
(a) Consider a graded minimal free resolution of $C$ as an $R$-module
\[
0 \to F_c \to \cdots \to F_1 \to R \to C \to 0,
\]
where $c = \dim R- \dim C$ denotes the codimension.
Dualizing and then shifting it, we obtain a graded minimal free resolution of $\omega_C$ of the form
\[
0 \to R(-\dim R) \to F_1^*(-\dim R) \to  \cdots \to  F_c^* (-\dim R)  \to \omega_C \to 0. 
\]
Claim (a) follows using the characterization of regularity by graded Betti numbers. 

(b) The $\Tor$ sequence of the given short exact sequence begins 
\[
0 \to \Tor_c^R (C, K) \to \Tor_{c-1}^R (\omega_{\tC}, K) (-t) \to \cdots .
\] 
Since $C$ is a Gorenstein algebra,  $\Tor_c^R (C, K)$ is concentrated in degree $c + \reg C$. The above resolution of $\omega_C$ shows that  $\Tor_{c-1}^R (\omega_{\tC}, K) (-t)$ is concentrated in degree $t+\dim R$.
It follows that $c + \reg C = t +\dim R$, which proves Claim (b). 
\end{proof}

\section{Graded Betti numbers of the Connected Sum} \label{sec:Betti numbers}
\subsection{Two Summands}

Let $R = K[x_1,\ldots,x_m]$ and $S = K[y_1,\ldots,y_n]$ be polynomial rings over  $K$ with the standard grading, and let $\x = (x_1,\ldots,x_m)$ and $\y = (y_1,\ldots,y_n)$ denote the homogeneous maximal ideals of $R$ and $S$, respectively. Set  
\[
Q = R \otimes_K S \cong K[x_1,\ldots,x_m, y_1,\ldots,y_n].
\]

Let $A = R/\fa$ and $B = S/\fb$ be standard graded  $K$-algebras. 
We will assume $\fa_1 = \fb_1 = 0$, that is, the ideals $\fa$ and $\fb$ do not contain any non-zero linear forms.  Note that $R\times_K S=Q/(\x\cap \y)$ by Lemma \ref{lem:fiber product over K}. We start by determining the Betti numbers of this ring. 
Note that the ideal $\x \cap \y$ of $Q$ is a so-called {\em Ferrers ideal} and admits a minimal graded free cellular resolution that is supported on the join of two simplices (see \cite{CN}).
We show other useful fact about this resolution below. 

Henceforth, we set $\binom{a}{b}=0$ if $b>a$.

\begin{lemma}
   \label{lem:Betti product ideal}
The ideal $\x \cap \y$ of $Q$ has a 2-linear minimal free graded resolution over $Q$ and, for $i \ge 1$, we have 
\begin{multline*}
[\Tor_i^Q( Q/\x \cap \y, K)]_{i+1}\cong  \\
\coker \big ( [\Tor^Q_{i+1} (Q/\x, K) \oplus \Tor^Q_{i+1} (Q/\y, K)]_{i+1} \to [\Tor^Q_{i+1} (K, K)]_{i+1}  \big ). 
\end{multline*}
In particular, one has 
\begin{equation*} 
   \label{eq:explict Betti product}
\dim_K [\Tor_i^Q( Q/\x \cap \y, K)]_{i+1} = \binom{m+n}{i+1} - \binom{m}{i+1} - \binom{n}{i+1} 
\end{equation*}
and $P^Q_{Q/\x \cap \y}(t,s)
=t^{-1}[(1+st)^{m}-1][(1+st)^n-1]+1$.
\end{lemma}

\begin{proof}
For brevity, let us write $\Tor_i^Q (M)$ instead of $\Tor_i^Q (M, K)$ for a $Q$-module $M$. The first part is well-known (see, e.g., \cite[Theorem 2.1]{CN}). It also follows from the Mayer-Vietoris sequence
\[
0 \to Q/\x \cap \y \to Q/\x \oplus Q/\y \to Q/\x+\y \to 0. 
\]
Observe that $Q/\x+\y \cong K$ as a $Q$-module. Thus, for any integer $i \ge 0$, the induced long exact sequence in $\Tor$ gives 
\begin{gather*} 
[\Tor_{i+1}^Q( Q/\x \cap \y)]_{i+1} \to   [\Tor_{i+1}^Q (Q/\x) \oplus \Tor_{i+1}^Q (Q/\y)]_{i+1} \to \\
[\Tor_{i+1}^Q (K)]_{i+1} 
\to [\Tor_i^Q( Q/\x \cap \y)]_{i+1} \to [\Tor_{i}^Q (Q/\x) \oplus \Tor_{i}^Q (Q/\y)]_{i+1}. 
\end{gather*}
The left-most module in this sequence is zero because $\x \cap \y$ has a 2-linear resolution. Similarly, the right-most module is zero because $\x$ and $\y$ have linear resolutions. Thus, the second claim follows.
The last claim can be verified directly based on the second. 
\end{proof}

The following result gives a formula to compute the graded Betti numbers of the fiber product $A\times_K B$. 

\begin{notation}Given a power series $P(t,s)=\sum c_{ij}s^jt^i\in\mathbb{Z}[s][[t]]$, we set $\widetilde{P}(t,s)=\sum\limits_{j>i} c_{ij}s^jt^i$ to be the sum of the terms of $P(t,s)$ having $j>i$.
\end{notation}

\begin{theorem}
    \label{thm:Betti fiber product}

For any integer $i\ge 1$, there is an isomorphism of graded $K$-vector spaces
\begin{multline*} \label{eq:Betti fiber product}
 [\Tor_i^Q (A \times_K B, K)]_j \\
\cong \begin{cases}
0 & \text{ if $j \le i$}, \\[2pt]
[\Tor_i^Q (A , K)]_{i+1} \oplus  [\Tor_i^Q (B , K)]_{i+1}  & \\
 \oplus   [\Tor_i^Q( Q/\x \cap \y, K)]_{i+1} & \text{ if $j = i+1$}, \\[2pt]
  [\Tor_i^Q (A , K)]_{j} \oplus [\Tor_i^Q (B , K)]_{j}  & \text{ if $j \ge i+2$}. 
\end{cases}
\end{multline*}
In particular, one has 
\[
P^Q_{A \times_K B}(t,s)=\widetilde{P}^Q_{A} (t,s)+\widetilde{P}^Q_{B}(t,s) + P^Q_{Q/\x \cap \y}(t,s) 
\]and
\[
\reg (A \times_K B) = \max\{\reg A ,\reg B\}.
\]
\end{theorem}

\begin{proof}
 Again, we write $\Tor_i^Q (M)$  instead of $\Tor_i^Q (M, K)$ for any $Q$-module $M$. 
Consider the exact sequence \eqref{exactFP} of graded $Q$-modules
\[
0 \to A \times_K B \to A \oplus B \to K \to 0. 
\]
Its long exact sequence in $\Tor$ gives exact sequences 
\begin{equation}
   \label{tor sequence fiber}
\Tor_{i+1}^Q (K) \to \Tor_i^Q (A \times_K B) \to \Tor_i^Q (A) \oplus \Tor_i^Q (B) \to \Tor_i^Q (K). 
\end{equation} 

As a $Q$-module, $K$ is resolved by a Koszul complex, which shows in particular $[\Tor_i^Q (K)]_j \neq 0$ if and only if $0 \le i = j \le m+n$. Considering the sequence \eqref{tor sequence fiber} in degree $j$, we conclude that 
\[
[\Tor_i^Q (A \times_K B, K)]_j  \cong [\Tor_i^Q(A , K)]_{j} \oplus [\Tor_i^Q (B , K)]_{j}  
\]
if $j \ge i+2$. 

Using the fact that $A \times_K B \cong Q/(\x \cap \y, \fa, \fb)$--as in \eqref{eq:fiber product presentation}--and that the initial degree of $(\x \cap \y, \fa, \fb)$ is two since $\fa$ and $\fb$ do not contain any linear forms by assumption, we see that for $j \le i$,
\[
[\Tor_i^Q(A \times_K B, K)]_j  =0 .  
\]

It remains to consider $[\Tor_i^Q (A \times_K B, K)]_{i+1}$. To this end, we use a longer part of the exact $\Tor$ sequence and consider it in degree $i+1$: 
\begin{align*}
& [\Tor_{i+1}^Q (A) \oplus \Tor_{i+1}^Q (B)]_{i+1}   \stackrel{\gamma}{\longrightarrow} [\Tor_{i+1}^Q (K)]_{i+1} \to\\
&  \hspace*{1cm} [\Tor_i^Q (A \times_K B)]_{i+1} \to [\Tor_i^Q (A) \oplus \Tor_i^Q (B)]_{i+1} \to [\Tor_i^Q (K)]_{i+1}. 
\end{align*}
The right-most module in this sequence is zero because $K$ has a linear resolution as a $Q$-module.
Note that there are isomorphisms of graded $Q$-modules $A \cong Q/(\fa, \y)$ and $B = Q/(\x, \fb)$. Since $\fa_1 = \fb_1 = 0$, it follows that 
\[
[\Tor_{i+1}^Q (A) \oplus \Tor_{i+1}^Q (B)]_{i+1} \cong [\Tor_{i+1}^Q (Q/\y) \oplus \Tor_{i+1}^Q (Q/\x)]_{i+1}.\]
Therefore, the cokernel of the map $\gamma$ is equal to the cokernel of the map $ [\Tor_{i+1}^Q (Q/\x) \oplus \Tor_{i+1}^Q (Q/\y)]_{i+1} \to [\Tor_{i+1}^Q (K)]_{i+1}$. By \Cref{lem:Betti product ideal}, the latter is isomorphic to $[\Tor_i^Q( Q/\x \cap \y, K)]_{i+1}$. Hence, the above sequence proves the claim for $[\Tor_i^Q (A \times_K B, K)]_{i+1}$. 

The claim regarding the Poincar\'{e} series follows by taking the dimensions of the $\Tor$ modules and using the isomorphism in the first part of the theorem. 
\end{proof}

\begin{remark}
The previous result allows us to write the Betti numbers of the fiber product in terms of those of the summands:
\begin{equation*} 
\beta^Q_{i,j} (A \times_K B) =
    \begin{cases}
    0 & \text{ if } j \le i \\[2pt]
    \beta^Q_{i,i+1} (A) + \beta^Q_{i,i+1} (B) + \beta^Q_{i,i+1} ( Q/\x \cap \y) & \text{ if } j = i+1 \\[2pt]
     \beta^Q_{i,j} (A) + \beta^Q_{i,j} (B)  & \text{ if } j \ge i+2. 
    \end{cases}
\end{equation*}
\end{remark}

In \Cref{thm:Betti fiber product}, we have  computed the graded Betti numbers of
the fiber product $A\times_K B$  in terms of the graded Betti numbers of $A$ and $B$ as $Q$-modules and we will now convert these formulas into formulas depending only on $[\Tor_i^R (A, K)]_j$ and $[\Tor_i^S (B, K)]_j$.

\begin{notation}
We set $[x]_+=\max\{0,x\}$. 
\end{notation}

\begin{proposition}\label{betaTbetaR} 
The identities 
\begin{eqnarray*}
  P^Q_A(t,s) &=& P^R_A(t,s)\cdot (1+st)^{\dim S}, \text{ and} \\
  P^Q_B(t,s) &=& P^S_B(t,s)\cdot (1+st)^{\dim R}
\end{eqnarray*}  yield
\begin{eqnarray*}
  \widetilde{P}^Q_A(t,s) &=& (P^R_A(t,s)-1)\cdot (1+st)^{\dim S}, \text{ and} \\
  \widetilde{P}^Q_B(t,s) &=& (P^S_B(t,s)-1)\cdot (1+st)^{\dim R}.
\end{eqnarray*}
Thus, with $m=\dim R$ and $n=\dim S$, we have
\begin{eqnarray*}
\beta _{i,j}^Q(A) &=& \sum _{\ell =0}^{\min(i,m)-1} \binom{n}{[i-m]_++\ell}\beta_{\min(i,m)-\ell,j-\ell-[i-m]_+}^R(A), \text{ and} \\
\beta _{i,j}^Q(B) &=& \sum _{\ell =0}^{\min(i,n)-1} \binom{m}{[i-n]_++\ell}\beta_{\min(i,n)-\ell,j-\ell-[i-n]_+}^S(B).
\end{eqnarray*}

\end{proposition}
\begin{proof}
Since $Q$ is a free $R$-module, we have $P^Q_{Q/\fa}(t,s)=P^R_{R/\fa}(t,s)=P^R_A(t,s)$. We know $A=Q/(\fa +\y)$, and $\Tor_i^K(Q/\fa,Q/\y)=0$ for $i\geq 1$, so the minimal free resolution of $A$ as a $Q$-module is obtained  by tensoring  the minimal free resolutions of $Q/\fa$ and $Q/\y$. This justifies the identity 
\[
P^Q_A(t,s)=P^Q_{Q/\fa}(t,s)\cdot P^Q_{Q/\y}(t,s)=P^R_A(t,s)\cdot (1+st)^{\dim S}.
\]
Since the terms with equal exponents for $s$ and $t$ in $P^Q_A(t,s)$ arise  from the constant term of the first factor multiplied with the second factor, removing this term yields $\widetilde{P}^Q_A(t,s)$.
Therefore, using the convention $\binom{n}{a}=0$ if $a>n$, we conclude that
$$
\beta _{i,j}^Q(A)=\begin{cases} \sum _{\ell =0}^{i-1} \binom{n}{\ell}\beta_{i-\ell,j-\ell}^R(A) \text{ if } 1\le i \le m, \text{ and } \\  \\
\sum _{\ell = 0}^{m-1}  \binom{n}{i-m+\ell}\beta_{m-\ell,j-i+m-\ell}^R(A)   \text{ if } m<i\le m+n.
\end{cases}
$$
The claims regarding $B$ are justified similarly. 
\end{proof}

Combining Lemma \ref{lem:Betti product ideal}, Theorem \ref{thm:Betti fiber product} and Proposition \ref{betaTbetaR}, we get:
\begin{corollary} With the above notation, we have 
\begin{equation}
\label{eq:Geller}
\begin{split}
    P^Q_{A \times_K B}(t,s)={}& (P^R_{A}(t,s)-1)\cdot(1+st)^{n}  +(P^S_{B}(t,s)-1)\cdot (1+st)^{m} +\\ 
    & +t^{-1}[(1+st)^{m}-1][(1+st)^n-1]+1,
\end{split}
\end{equation}
that is,
\begin{multline*}
\beta^Q_{i,j} (A \times_K B) =\\
    \begin{cases}
   0 & \text{ if } j \le i, (i,j)\neq(0,0), \\
 1 &\text{ if } i=j=0, \\[.25cm]
    \begin{array}{l}
    \sum _{\ell =0}^{\min(i,m)-1} \binom{n}{[i-m]_++\ell}\beta_{\min(i,m)-\ell,i+1-\ell-[i-m]_+}^R(A)+\\
    \sum _{\ell =0}^{\min(i,n)-1} \binom{m}{[i-n]_++\ell}\beta_{\min(i,n)-\ell,i+1-\ell-[i-n]_+}^S(B) \\
    +\binom{m+n}{i+1} - \binom{m}{i+1} - \binom{n}{i+1}
    \end{array}& \text{ if } j = i+1, \\[.5cm]
     \begin{array}{l} \sum _{\ell =0}^{\min(i,m)-1} \binom{n}{[i-m]_++\ell}\beta_{\min(i,m)-\ell,j-\ell-[i-m]_+}^R(A)\\
     +\sum _{\ell =0}^{\min(i,n)-1} \binom{m}{[i-n]_++\ell}\beta_{\min(i,n)-\ell,j-\ell-[i-n]_+}^S(B) \end{array} & \text{ if } j \ge i+2. 
    \end{cases}
\end{multline*}
\end{corollary}

\begin{remark}
The identity \eqref{eq:Geller} extends \cite[Theorem 1.1]{G} to the graded case. It can be checked that the two results agree upon substituting $s=1$ in \eqref{eq:Geller}.
\end{remark}

We now turn to the graded Betti numbers of the connected sum. AG $K$-algebras with socle degree two have been classified by \cite{S}. Explicitly, if $A = R/\fa$ has $h$-vector $(1, n, 1)$ then, up to isomorphism, one has 
\[
\fa = (x_i x_j \; \mid \; 1 \le i \neq j \le n) + (x_1^2 -  x_2^2,\ldots,x_1^2-x_n^2).
\]
The graded minimal free resolution of $A$ as $R$-module has the form 
\[
0 \to R(-n-2) \to R(-n)^{\beta_{n-1}} \to \cdots \to R(-2)^{\beta_1} \to R \to A \to 0, 
\]
and a straightforward computation gives us
 $$\beta_i=\beta _{n-1-i}= i\binom{n}{i+1}+(n-i)\binom{n}{n-i+1} $$ 
for $1\le i \le n-1$.
Thus, it is harmless to consider AG $K$-algebras whose socle degree is at least three. 

\begin{theorem}
    \label{thm:Betti conn sum}
Assume that $A$ and $B$ are AG $K$-algebras such that $\reg(A)=\reg(B)=e\ge 3$. 
For any integer $i\ge 1$, there is an isomorphism of graded $K$-vector spaces
\begin{multline*} 
   \label{eq:Betti conn sum}
[\Tor_i^Q (A \#_K B, K)]_j  \cong\\
\begin{cases}
0 & \text{ if } j \le i \text{ and } (i, j) \neq (0, 0), \\[2pt]
K & \text{ if } (i, j) = (0, 0), \\[2pt]
[\Tor_i^Q (A , K)]_{i+1} \oplus [\Tor_i^Q (B , K)]_{i+1}  & \\
 \oplus  [\Tor_i^Q( Q/\x \cap \y, K)]_{i+1} & \text{ if } j = i+1, \\[2pt]
 [\Tor_{i}^Q (A , K)]_{j} \oplus [\Tor_{i}^Q (B , K)]_{j}  & \text{ if } i+2 \le j \le i + e -2, \\[2pt]
 %
  [\Tor_{i}^Q (A , K)]_{j} \oplus [\Tor_{i}^Q (B , K)]_{j} & \\
 \oplus  [\Tor_{m+n-i}^Q( Q/\x \cap \y, K)]_{m+n-i+1}& \text{ if } j = i+ e - 1, \\[2pt] 
 K & \text{ if } (i, j) = (m+n, e+m+n), \\[2pt]
 0 & \text{ if } j \ge e + i \text{ and} \\
 & (i, j) \neq (m+n, e+m+n);
\end{cases}
\end{multline*}
equivalently,
\begin{multline*}
P^Q_{A\#_KB}(t,s)=\widetilde{P}_A^Q(t,s)+\widetilde{P}_B^Q(t,s)+P^Q_{Q/\x\cap\y}(t,s)\\
+s^{m+n+e}t^{m+n}P^Q_{Q/\x\cap \y}(t^{-1},s^{-1}).
\end{multline*}
\end{theorem}

\begin{proof}
For a $Q$-module $M$, we write $\Tor_i^Q(M)$  instead of $\Tor_i^Q (M, K)$. 
Consider the exact sequence \eqref{exactCS} of graded $Q$-modules
\[
0  \to K (- e) \to A \times_K B  \to A  \#_K B  \to 0. 
\]
Its long exact $\Tor$ sequence gives exact sequences 
\begin{multline*}
[\Tor_i^Q (K)]_{j-e} \to [\Tor_i^Q (A \times_K B)]_j \to\\
[\Tor_i^Q (A \#_K B)]_j \to [\Tor_{i-1}^Q (K)]_{j - e}. 
\end{multline*}
Since $\Tor_i^Q (K)$ is concentrated in degree $i$ we conclude that 
\[
[\Tor_i^Q (A \#_K B)]_j  \cong [\Tor_i^Q (A \times_K B)]_j  
\] 
if $j \notin \{e + i - 1, e + i\}$. Combined with \Cref{thm:Betti fiber product}, this determines $[\Tor_i^Q (A \#_K B)]_j$ if $j \le e + i-2$. 

Using the fact that  $\reg (A  \#_K B) = \reg A = \reg B=e$ which can be deduced from \eqref{HilbertCS}, we know that $[\Tor_i^Q (A \#_K B)]_j = 0$ if $j \ge e + i + 1$. It remains to determine $[\Tor_i^Q (A \#_K B)]_j$ if $j \in \{e + i - 1,e + i\}$. To this end we utilize the fact that $A \#_K B$ is Gorenstein. Thus, its graded minimal free resolution is symmetric. In particular, since $\dim Q=m+n$, one has
\[
[\Tor_i^Q (A \#_K B)]_j  \cong [\Tor_{m+n-i}^Q (A \#_K B)]_{e + m+n - j}.
\]
Similarly, for $A$ and $B$ we have
$
\Tor_i^Q (A)_j  \cong [\Tor_{m+n-i}^Q (A)]_{e + m+n - j}$  and 
$\Tor_i^Q (B)_j  \cong [\Tor_{m+n-i}^Q (B)]_{e + m+n - j}$.

Combined with \Cref{thm:Betti fiber product} and using $e \ge 3$, which implies that the degrees $e+i-1, e+i$ are not self-dual under the isomorphisms given above, the claim regarding $\Tor$ modules follows. 

The Poincar\'e series formula follows from the above considerations and the identities
\begin{eqnarray*}
&\sum_{i=0}^{m+n} \beta_{m+n-i, m+n-i+1}(Q/\x\cap\y)t^{i}s^{i+e-1} \\
=& \sum_{j=0}^{m+n} \beta_{j, j+1}(Q/\x\cap\y)  t^{m+n-j}s^{m+n-j+e-1}\\
=&t^{m+n}s^{m+n+e}\sum_{j=0}^{m+n} \beta_{j, j+1}(Q/\x\cap\y) t^{-j}s^{-j-1}\\
=&t^{m+n}s^{m+n+e} P^Q_{Q/\x\cap\y}(t^{-1},s^{-1}).
\end{eqnarray*}
\end{proof}

Using again Lemma \ref{lem:Betti product ideal} and Proposition \ref{betaTbetaR}, we can convert Theorem \ref{thm:Betti conn sum} into explicit formulas depending only on 
the Betti numbers of $A$ as an $R$-module and of $B$ as an $S$-module.

\begin{corollary} With the above notation we have:

\hspace{2em}\hspace{-2em}
   \label{eq:Betti conn sum}
$\beta _{i,j}^Q(A  \#_K B)= \nonumber\\
  \begin{cases}
0 & \text{ if } j\le i, (i, j) \neq (0, 0), \\[.4cm]
1 & \text{ if } (i, j) = (0, 0), \\[2pt]
 \sum _{\ell =0}^{min(i,m)-1} \binom{n}{[i-m]_+ +\ell}\beta_{min(i,m)-\ell,i+1-\ell-[i-m]_+}^R(A)\\+ \sum _{\ell =0}^{min(i,n)-1} \binom{m}{[i-n]_+ +\ell}\beta_{min(i,n)-\ell,i+1-\ell-[i-n]_+}^S(B) \\ +\binom{m+n}{i+1} - \binom{m}{i+1} - \binom{n}{i+1}
     & \text{ if } j = i+1, \\[.4cm]
\sum _{\ell =0}^{min(i,m)-1} \binom{n}{[i-m]_+ +\ell}\beta_{min(i,m)-\ell,j-\ell-[i-m]_+}^R(A)\\+ \sum _{\ell =0}^{min(i,m)-1} \binom{m}{[i-n]_+ +\ell}\beta_{min(i,n)-\ell,j-\ell-[i-n]_+}^S(B)   & \text{ if } i+2 \le j \le i + e -2, \\[.5cm]
 \sum _{\ell =0}^{min(i,m)-1} \binom{n}{[i-m]_+ +\ell}\beta_{min(i,m)-\ell,i+1-\ell-[i-m]_+}^R(A)\\+ \sum _{\ell =0}^{min(i,n)-1} \binom{m}{[i-n]_+ +\ell}\beta_{min(i,n)-\ell,i+1-\ell-[i-n]_+}^S(B) \\+ \binom{m+n}{m+n-i-1} - \binom{m}{m-i-1} - \binom{n}{n-i-1} & \text{ if } i\leq m+n, j = i+ e - 1, \\[2pt] 
 1 & \text{ if } (i, j) = (m+n, e+m+n), \\[.4cm]
 0 & \text{ if } j \ge e + i \text{ and} \\
 & (i, j) \neq (m+n, e+m+n).
\end{cases}
$
\end{corollary}

\subsection{Arbitrary many summands}
For $i = 1,\ldots,r$, consider standard graded polynomial rings $R_i=K[x_{i,1},\ldots, x_{i,n_i}]$ with irrelevant maximal ideals $\x_i = (x_{i,1},\ldots, x_{i,n_i})$.
Also, let $Q =  R_1 \otimes_K \cdots \otimes_K R_r$. In the following, we abuse notation to write 
$\x_i$ to also denote the extensions of these ideals to ideals of $Q$.

\begin{lemma}
   \label{lem:annihilator}
Consider ideals $I_1,\ldots,I_r$ of a commutative ring $P$ and the map $\varphi \colon P/\bigcap_{j=1}^r I_j \stackrel{\varphi}{\longrightarrow} \bigoplus_{j=1}^r P/I_j$, where $\varphi$ maps the image of $p \in P$  in $P/\bigcap_{j=1}^r I_j$ onto the image of $(p,\ldots,p) \in P^r$ in $\bigoplus_{j=1}^r P/I_j$. The annihilator of $\coker \varphi$ as a $P$-module is $\bigcap_{i=1}^r \big ( I_i + \bigcap_{j \neq i} I_j \big )$. 

\end{lemma} 

\begin{proof}
First, we prove the inclusion $\Ann (\coker \varphi) \subseteq \bigcap_{i=1}^r \big ( I_i + \bigcap_{j \neq i} I_j \big )$. Consider $a \in \Ann (\coker \varphi)$ and let $m$ be the image of $(1,0,\dots,0)\in P^r$ in $\bigoplus_{j=1}^r P/I_j$, so that $am \in \im \varphi$. Then there exists $p\in \bigcap_{j\neq 1} I_j$ such that $a-p\in I_1$, which shows that $a\in I_1 +\bigcap_{j\neq 1} I_j$. Repeating the argument with $2\leq i\leq r$ proves the desired inclusion.

For the reverse inclusion, observe that every $a\in I_1 +\bigcap_{j\neq 1} I_j$ annihilates the image of $(1,0,\dots,0)\in P^r$ in $\coker \varphi$, and similarly for $2\leq i\leq r$.
\end{proof} 

\Cref{lem:annihilator} will be applied to the following family of ideals.

\begin{lemma}
   \label{lem:special ideals}
Define the ideals $I_j$ of $Q$ by 
\[
I_j = \x_1 + \cdots + \widehat{\x}_j + \cdots + \x_r
\] 
so that $Q/I_j \cong R_j$. Then one has
\[
\sum_{1\leq i \neq j\leq r} \x_i \cap \x_j = \bigcap_{j=1}^r I_j. 
\]
\end{lemma} 

\begin{proof}
We proceed by induction on $r \ge 2$, the base case follows immediately from the definitions.

Assume the statement holds for $r-1$, i.e.,
\[\sum_{1\leq i \neq j\leq r-1} \x_i \cap \x_j = \bigcap_{j=1}^{r-1} \tilde{I}_j,\]
where $\tilde{I}_j = \x_1 + \cdots + \widehat{\x}_j + \cdots + \x_{r-1}$. Then we have
\begin{equation*}
    \begin{split}
        &\bigcap_{j=1}^r I_j = \bigcap_{j=1}^{r-1} I_j \cap I_r =
        \bigcap_{j=1}^{r-1} (\tilde{I}_j +\x_r) \cap I_r =
        \left(\bigcap_{j=1}^{r-1} \tilde{I}_j +\x_r \right) \cap I_r \\
        =&\left(\sum_{1\leq i \neq j\leq r-1} \x_i \cap \x_j +\x_r \right) \cap I_r =
        \sum_{1\leq i \neq j\leq r-1} \x_i \cap \x_j
        +\sum_{1\leq i \leq r-1} \x_i \cap \x_r
    \end{split}
\end{equation*}
because $\x_i \cap \x_j \subset I_r$ whenever $1\leq i\neq j\leq r-1$.
\end{proof}
Combining \Cref{lem:annihilator} with \Cref{lem:special ideals}, we obtain: 

\begin{corollary}\label{2-linear_cor}
There is an exact sequence of graded $Q$-modules 
\begin{equation}
\label{eq:2-linear_cor}
0 \to Q/\left(\sum_{1\leq i \neq j\leq r} \x_i \cap \x_j\right) \to \bigoplus_{j=1}^r R_j \to K^{r-1} \to 0. 
\end{equation}
\end{corollary} 

\begin{proof} 
Consider the ideal $I_1,\ldots,I_r$ of $Q$ as defined in \Cref{lem:special ideals}. Observe that $Q/I_j \cong R_j$ and $\bigcap_{j=1}^r I_j = \sum_{1\leq i \neq j\leq r} \x_i \cap \x_j$. Thus, the map in \Cref{lem:annihilator} becomes $\varphi \colon Q/\left(\sum_{1\leq i \neq j\leq r} \x_i \cap \x_j \right) \to \bigoplus_{j=1}^r R_j.$ 

The definition of the ideals $I_j$ implies that for each $i$, $I_i + \bigcap_{j \neq i} I_j$ is the maximal ideal $\fm$ generated by all the variables of $Q$. Hence, \Cref{lem:annihilator} shows that $\Ann (\coker \varphi) = \fm$. 

Notice that $\im \varphi$ is a cyclic $Q$-module whose minimal generator can also be taken as a minimal generator of $\bigoplus_{j=1}^r R_j$. It follows that $\coker \varphi$ is minimally generated by $r-1$ elements of degree zero. Since $\Ann (\coker \varphi) = \fm$, we conclude that $\coker \varphi \cong K^{r-1}$, which completes the argument. 
\end{proof}


We compute the Betti numbers for the leftmost term in the short exact sequence \eqref{eq:2-linear_cor}.

\begin{lemma}\label{lem:Betti product ideal_many summands}
    The ideal $\sum_{i \neq j} \x_i \cap \x_j $ has a 2-linear minimal free resolution, and for $t\geq 1$ we have that 
\begin{align*} 
\hspace{2em}&\hspace{-2em}
\left[\Tor_t^Q( Q/\sum_{i \neq j} \x_i \cap \x_j, K) \right]_{t+1}\cong  \\
& \coker \left ( \left[\bigoplus_{j=1}^r\Tor^Q_{t+1} (Q/I_j, K) \right]_{t+1} \to \left[\bigoplus_{j=1}^{r-1}\Tor^Q_{t+1} (K, K) \right]_{t+1}  \right ).
\end{align*}
Thus with $N=n_1+\cdots+n_r$,
\[
\beta_{t,t+1}^Q\left(Q/\sum_{i \neq j} \x_i \cap \x_j\right)=(r-1)\binom{N}{t+1}-\sum_{k=1}^r\binom{N-n_k}{t+1}.
\]
\end{lemma}
\begin{proof}
Since $\reg(Q/I_j)=0$ for each $1\leq j\leq r$ and $\reg(K)=0$, 
    the short exact sequence \eqref{eq:2-linear_cor} implies, by means of the formula 
    \[
    \reg\left( Q/\sum_{i \neq j} \x_i \cap \x_j \right) \leq \max \left\{ \reg( \bigoplus_{j=1}^r  Q/I_j), \reg(K^{r-1})+1\right \}=1,
    \]
 that $\sum_{i \neq j} \x_i \cap \x_j $ has a 2-linear minimal free resolution. Moreover,  for every $t\geq 0$, it induces the following long exact sequence. Note that we write $\Tor^Q_t(M)$ instead of $\Tor^Q_t(M,K)$ for a $Q$-module $M$.
\begin{multline*} 
   \left[\Tor_{t+1}^Q \left( Q/\sum_{i \neq j} \x_i \cap \x_j \right) \right]_{t+1}
   \to \left[\bigoplus_{j=1}^r\Tor^Q_{t+1} (Q/I_j) \right]_{t+1} \\
   \to \left[\bigoplus_{j=1}^{r-1}\Tor^Q_{t+1} (K) \right]_{t+1}
   \to \left[\Tor_{t}^Q \left( Q/\sum_{i \neq j} \x_i \cap \x_j \right) \right]_{t+1} \\
   \to \left[\bigoplus_{j=1}^r\Tor^Q_{t} (Q/I_j) \right]_{t+1}.
\end{multline*}
Since  $\reg (\sum_{i \neq j} Q/\x_i \cap \x_j ) = 1$ and $\reg(Q/I_j)=0$, the left-most and right-most modules in the above long exact sequence are zero. Since the $Q$-modules $Q/I_j$ and $K$ are minimally resolved by Koszul complexes, the dimensions of the second and third terms of the sequence are given by sums of the appropriate binomial coefficients. Therefore, taking the difference of these dimensions yields the desired formula for the Betti numbers.
\end{proof}

For every $i=1,\dots ,r$ we consider a standard graded ring $A_i = R_i/\fa_i$ with $\fa_i\subseteq (\x_i)^2$. 
We abuse notation to write $\fa_i$ to also denote the extensions of these ideals to ideals of $Q$.

\begin{theorem}\label{thm:tors multifactor FP}
    For every $t\geq 1$, we have that 
\begin{equation*}
    \begin{split}
        &[\Tor_t^Q (A_1\times_K\cdots \times_KA_r, K)]_s =\\
        &\begin{cases}
        0 & \text{ if } s \le t, \\[2pt]
        \displaystyle\bigoplus_{i=1}^r[\Tor_t^Q (A_i , K)]_{t+1} \oplus
        \left[\Tor_t^Q \left( Q/\sum_{i\neq j} \x_i \cap \x_j, K \right) \right]_{t+1} & \text{ if } s = t+1, \\[2pt]
        \displaystyle\bigoplus_{i=1}^r[\Tor_t^Q (A_i , K)]_{s}   & \text{ if } s \ge t+2. 
        \end{cases}
    \end{split}
\end{equation*}
\end{theorem}
\begin{proof}
    Consider the short exact sequence of graded $Q$-modules \eqref{exactFPr}
    \[
0 \to A_1\times_K\cdots \times_KA_r \to A_1\oplus\cdots \oplus A_r \to K^{r-1} \to 0. 
    \]
    For every $t\geq 0$, it induces the following long exact sequence, where we write $\Tor^Q_t(M)$ instead of $\Tor^Q_t(M,K)$ for a $Q$-module $M$. 
\begin{multline*} 
   \Tor^Q_{t+1}(K^{r-1}) \to \Tor^Q_t(A_1\times_K\cdots \times_KA_r)\\
   \to \Tor^Q_{t}(A_1\oplus\cdots \oplus A_r)\to\Tor^Q_t(K^{r-1}).
\end{multline*} 
We have that $[\Tor^Q_{t}(K^{r-1})]_s\neq 0$ if and only if $0\le t=s\le n_1+\cdots +n_r$. Thus, for every $s\ge t+2$, we get 
\[
[\Tor_t^Q (A_1\times_K\cdots \times_KA_r)]_s\cong\bigoplus_{i=1}^r[\Tor_t^Q (A_i)]_{s}.
\]
Restricting to degree $s=t+1$, we get the exact sequence:
    \begin{multline*} 
   \bigoplus_{i=1}^{r}[\Tor^Q_{t+1}(A_i)]_{t+1}\to [\Tor^Q_{t+1}(K^{r-1})]_{t+1}\\
   \to [\Tor_t^Q(A_1\times_K\cdots \times_KA_r)]_{t+1} \to \bigoplus_{i=1}^r[\Tor^Q_t(A_i)]_{t+1}\to 0.
\end{multline*}
    For every $1\le i\le r$, we have $A_i=Q/(I_i,\fa_i)$, and since we assume that each $\fa_i$ is generated in degrees at least two, we also have  
    \[
\bigoplus_{i=1}^r[\Tor^Q_{t+1}(A_i)]_{t+1}\cong \bigoplus_{i=1}^r[\Tor^Q_{t+1}(Q/I_i)]_{t+1}.
    \]
 This implies that 
  \begin{multline*}
[\Tor^Q_{t}(A_1\times_K\cdots \times_K A_r)]_{t+1}\cong \bigoplus_{i=1}^r[\Tor^Q_t(A_i)]_{t+1}\\
\oplus \coker \left(\bigoplus_{i=1}^r[\Tor^Q_{t+1}(Q/I_i)]_{t+1}\to \bigoplus_{i=1}^{r-1}[\Tor^Q_{t+1}(K)]_{t+1}\right).
  \end{multline*}
Using Lemma \ref{lem:Betti product ideal_many summands}, we get the desired formula.
    \end{proof}

From the previous result, we can compute the graded Betti numbers of the fiber product $A_1\times_K\cdots \times_K A_r$ in terms of the graded Betti numbers of the $A_i$ as $Q$-modules. A straightforward computation allows us to translate into a formula depending only on the Betti numbers of the $A_i$ as $R_i$-modules.

\begin{corollary}\label{cor:Poincare multifactor FP over R}
  With $N=n_1+\cdots+n_r$, we have
\begin{eqnarray*}
    P^Q_{A_1\times_K\cdots \times_KA_r}(t,s) &=&\sum_{i=1}^r(P^{R_i}_{A_i} (t,s)-1)(1+st)^{N-n_i}\\
    &&+(r-1)\frac{(1+st)^{N}-Nst -1}{t}\\
  && -\sum_{i=1}^r\frac{(1+st)^{N-n_i}-(N-n_i)st-1}{t}+1.
\end{eqnarray*}
\end{corollary}
\begin{proof}
 The generating series is derived from Theorem \ref{thm:tors multifactor FP}, the numerical formula in Lemma \ref{lem:Betti product ideal_many summands}, and an analogue of Proposition \ref{betaTbetaR}.
\end{proof}

Recall that it is harmless to assume that an AG $K$-algebra has socle degree at least three because AG algebras with smaller socle degrees are well understood (see the description above \Cref{thm:Betti conn sum}). 

\begin{theorem}
    \label{thm:Betti multifactor conn sum}
Assume that $A_1, \ldots, A_r$ are AG $K$-algebras with $\reg(A_\ell)=e\ge 3$ for all $1\leq \ell\leq r$. Then, for any integer  $i \ge 0$, the graded Betti numbers of the connected sum $A_1 \#_K \cdots \#_K A_r$ over the polynomial ring $Q$ with $\dim(Q)=N$ are given  by 
\begin{align*} 
   \label{eq:Betti conn sum}
&[\Tor_t^Q (A_1 \#_K \cdots \#_K A_r, K)]_s \cong \nonumber\\
 & \begin{cases}
0 & \text{ if } s \le t \text{ and } (t, s) \neq (0, 0), \\[2pt]
K & \text{ if } (t, s) = (0, 0), \\[2pt]
\bigoplus_{i=1}^r[\Tor_t^Q (A_i , K)]_{t+1}  \\ \oplus  [\Tor_t^Q( Q/\sum_{i\neq j} \x_i \cap \x_j, K)]_{t+1} &\text{ if } j = i+1, \\[2pt]
 \bigoplus_{i=1}^r[\Tor_t^Q (A_i , K)]_{s}  & \text{ if } t+2 \le s \le t + e -2,  \\[2pt]
 \bigoplus_{i=1}^r[\Tor_t^Q (A_i , K)]_{s}  & \\
 \oplus  [\Tor_{N-t}^Q(  Q/\sum_{i\neq j} \x_i \cap \x_j, K)]_{N-t+1}& \text{ if } s = t+ e - 1, \\[2pt] 
 K & \text{ if } (t, s) = (N, e+N), \\[2pt]
 0 & \text{ if } s \ge e + t \text{ and} \\
 & (t, s) \neq (N, e+N).
\end{cases}
\end{align*}
\end{theorem}

\begin{proof}
As before, we write $\Tor_t^Q(M)$  instead of $\Tor_t^Q (M, K)$ for any $Q$-module $M$. 
We denote $C=A_1  \#_K \cdots \#_K A_r$ and $D= A_1 \times_K \cdots \times_K A_r$,
and consider the exact sequence  of graded $Q$-modules \eqref{eq:ses multifactor CS}
\[
0  \to K^{r-1} (- e) \to D  \to C  \to 0. 
\]
Its long exact $\Tor$ sequence gives exact sequences 
\[
[\Tor_t^Q (K^{r-1})]_{s - e} \to [\Tor_t^Q (D)]_s \to [\Tor_t^Q (C)]_s \to [\Tor_{t-1}^Q (K^{r-1})]_{s - e}. 
\]
Since $\Tor_t^Q (K)$ is concentrated in degree $t$ we conclude that 
\[
[\Tor_t^Q (C)]_s  \cong [\Tor_t^Q (D)]_s  
\] 
if $s \notin \{e + t - 1, e + t\}$. Combined with \Cref{thm:tors multifactor FP},
this determines $[\Tor_t^Q (C)]_s$ if $s \le e + t-2$. 

Using that  $\reg (C) = \reg A_i =e$, we know  $[\Tor_t^Q (C)]_s = 0$ if $s \ge e + t + 1$. It remains to determine $[\Tor_t^Q (C)]_s$ if $s \in \{e + t - 1,e + i\}$. To this end, we utilize the fact that $C$ is Gorenstein. Thus, its graded minimal free resolution is symmetric. In particular, one has
\[
[\Tor_t^Q (C)]_s  \cong [\Tor_{N-t}^Q (C)]_{e + N - s}
\]
and similarly, we have
$\Tor_t^Q (A_i)_s  \cong [\Tor_{N-t}^Q (A)]_{e + N - s}$
for each $A_i$.

Combined with \Cref{thm:tors multifactor FP} and using $e \ge 3$, which implies that the degrees $e+i-1, e+i$ are not self-dual under the isomorphisms given above, the claim regarding the $\Tor$ modules follows. 
\end{proof}

\begin{corollary}
    With the notation of \Cref{thm:Betti multifactor conn sum}, we have
  \begin{equation*}
  \begin{split}
  P^Q_{A_1\#_K\cdots \#_K A_r}(t,s) &=\sum_{i=1}^r(P^{R_i}_{A_i} (t,s)-1)(1+st)^{N-n_i} +1\\
    &+(r-1)\frac{(1+st)^{N}-Nst -1}{t}\\
  & -\sum_{i=1}^r\frac{(1+st)^{N-n_i}-(N-n_i)st-1}{t} \\
  &+(r-1)s^{N+e}t^{N+1}[(1+s^{-1}t^{-1})^{N}-\frac{N}{st} -1]\\
  & -s^{N+e}t^{N+1}\sum_{i=1}^r \left[(1+s^{-1}t^{-1})^{N-n_i}-\frac{N-n_i}{st}-1 \right].
  \end{split}
  \end{equation*}
\end{corollary}
\begin{proof}
As a first step, we show
\begin{equation*}
\begin{split}
P^Q_{A_1\#_K\cdots \#_K A_r}(t,s)
&=\sum_{i=1}^rP^Q_{A_i}(t,s)+P^Q_{ Q/\sum_{i\neq j} \x_i \cap \x_j}(t,s)\\
&+s^{N+e}t^{N}P^Q_{ Q/\sum_{i\neq j} \x_i \cap \x_j}(t^{-1},s^{-1})-2.
\end{split}
\end{equation*}
    This formula follows from Theorem \ref{thm:Betti multifactor conn sum} and the identities
\begin{eqnarray*}
&&\sum_{u=0}^{N} \beta_{N-u, N-u+1} \left( Q/\sum_{i\neq j}\x_i \cap \x_j \right)t^{u}s^{u+e-1} \\
&=& \sum_{v=0}^{N} \beta_{v, v+1} \left( Q/\sum_{i\neq j}\x_i \cap \x_j \right)  t^{N-v}s^{N-v+e-1}\\
&=&t^{N}s^{N+e}\sum_{v=0}^{N} \beta_{v, v+1} \left( Q/\sum_{i\neq j}\x_i \cap \x_j \right) t^{-v}s^{-v-1}\\
&=&t^{N}s^{N+e} P^Q_{ Q/\sum_{i\neq j}\x_i \cap \x_j}(t^{-1},s^{-1}).
\end{eqnarray*}
Substituting the formulas of \Cref{cor:Poincare multifactor FP over R} and \Cref{lem:Betti product ideal_many summands} into the formula above yields the claim.
\end{proof}



\section{Connected Sum as a Doubling} 
\label{sec:doubling} 
\subsection{Motivating examples} We discuss examples of monomial complete intersections. Using the so-called doubling method, Celikbas, Laxmi and Weyman  solved a particular case of Questions \ref{Q1} and \ref{Q2}. Indeed, in \cite[Corollary 6.3]{CLW}, they determine a minimal free resolution of the connected sum of $K$-algebras $A_i=K[x_i]/(x_i^{d_i})$ by using the doubling construction. The goal of this section is to generalize their result to AG $K$-algebras with the same socle degree. We start with a toy example.

\begin{example}\label{ex3} The Betti table of the connected sum
$$
C=\frac{K[x]}{(x^4)}\# _K\frac{K[y]}{(y^4)}\# _K\frac{K[z]}{(z^4)}
$$
is described on the left in \Cref{tab: betti 3}.
\begin{table}[ht]
    \centering 
    $$
\begin{tabular}{c|cccccc}
        & 0 & 1 & 2&3\\
       \hline
       total & 1 & 3 & 3& 1\\
       \hline
       0: & 1 & .  & .& . \\
       1: & . & 3 & 2 &. \\
       2:& . &  2&3&.\\
              3: &.&.&.&1
    \end{tabular}
\qquad
     \begin{tabular}{c|cccccc}
        & 0 & 1 & 2&3\\
       \hline
       total & 1 & 3 & 3& 1\\
       \hline
       0: & 1 & .  & .& . \\
       1: & . & 3 & 2 &. \\
\end{tabular}
$$  
\caption{Betti tables of $R/I$ and $R/J$ in \Cref{ex3}}
    \label{tab: betti 3}
\end{table}
It should be understood as follows. The connected sum $C$ has the presentation
\[
C=\frac{K[x,y,z]}{(xy, xz,yz, x^3+y^3, x^3+z^3)}.
\]
Let $Q = K[x, y, z], I=(xy, xz,yz, x^3+y^3, x^3+z^3)$ and $J = (xy, xz, yz)$. Then the Betti table of $Q/J$ is given on the right in \Cref{tab: betti 3}. It follows from this that $ \omega_{Q/J}$ has two generators and
there is an exact sequence 
$$ 0\to \omega_{Q/J}(-3)\to Q/J\to C\to 0,$$ 
which maps the generators of $ \omega_{Q/J}$ to the elements $x^3+y^3$ and $x^3+z^3$ in $Q/J$.
The resolution of $C$ is obtained as a mapping cone from the previous exact sequence. 

Each of the summands in $C$ is obtained by doubling a polynomial ring. Indeed, the short exact sequence
\[
0 \to \omega_{K[x]}(-3)\to K[x]\to \frac{K[x]}{(x^4)} \to 0
\]
sending the generator of $\omega_{K[x]}\cong K[x] (-1)$ to $x^4$, shows that $K[x]/(x^4)$ is a doubling of $K[x]$. Similarly, the remaining summands are doublings of $K[y]$ and $K[z]$, respectively. Furthermore, the ring $Q/J$ from above can be identified with the fiber product of the rings being doubled
\[
Q/J=K[x]\times_K K[y]\times_K K[z].
\]
\end{example}

The following example is the first generalization of the \cite[Corollary 6.3]{CLW} to every monomial complete intersections.
\begin{example}  
We focus on the connected sum $A=A_1 \# _K \cdots \# _KA_r$ of complete intersection algebras
\[
A_i:= K[x_{i,1},\dots,x_{i,n_i}] /(x_{i,1}^{d_{i,1}},\dots,x_{i,n_i}^{d_{i,n_i}})
\]
for $i=1,\dots,r$, satisfying 

\begin{equation}\label{eq: c}
 \sum_{j=1}^{n_i} d_{i,j}-n_i=\sum_{j=1}^{n_{i'}} d_{i',j}-n_{i'} \; \text{ whenever }1\leq i,i'\leq r.
 \end{equation}

Let $R_i=K[x_{i,1},\dots,x_{i,n_i}]$, $Q= R_1\otimes_K \cdots \otimes_K R_r$ and let $c$ be the quantity defined in \eqref{eq: c}. 
The connected sum of the $K$-algebras $A_i$ admits the presentation $A\cong Q /I$ where 
\begin{align*}
I &= \left(x_{i,j_i}x_{h,j_h} \,|\, 1\leq i<h\leq r,  1\leq j_i\leq n_i, 1\leq j_h\leq n_h\right)\\
&+ \left(x_{i,l_i}^{d_{i,l_i}} \,\middle|\, 1\leq i\leq r,  1\leq l_i\leq n_i-1\right)\\
&+ \left(x_{i,1}^{d_{i,1}-1}\cdots x_{i,n_i}^{d_{i,n_i}-1}+x_{1,1}^{d_{1,1}-1}\cdots x_{1,n_1}^{d_{1,n_1}-1} \,\middle|\, 2\leq i\leq r\right).
\end{align*}
It can be verified that $A$ is a doubling of $\tilde{A}=Q/J$, where $J$ is an ideal defining $r$ coordinate points in $\mathbb{A}_K^{n_1}\times \cdots \times \mathbb{A}_K^{n_r}$ with multiplicity; more precisely, $J=\bigcap_{i=1,\ldots,r}J_i$, where \begin{equation*}J_i=\left(x_{j,h},x_{i,l_i}^{d_{i,l_i}} \,\middle|\, j\neq i, 1\leq h\leq n_j,1\leq l_i\leq n_i-1\right).\end{equation*}
More importantly, setting $\tilde{A_i}=R_i/(x_{i,l_i}^{d_{i,l_i}} \,|\, 1\leq l_i\leq n_i-1)$, we see that each ring $A_i$ is a doubling of $\tilde{A_i}$ via the sequence

\[
0 \to \omega_{\tilde{A_i}}(-c)\to  \tilde{A_i}\to A_i\to 0
\]
sending the generator of $\omega_{\tilde{A_i}}(-c) \cong \tilde{A_i} (-d_{i,n_i})$ to $x_{i,n_i}^{d_{i,n_i}}$, and that $\tilde{A}=\tilde{A_1}\times_K\cdots \times_K\tilde{A_r}$. The Betti numbers of $\tilde{A}$ can thus be obtained via \Cref{cor:Poincare multifactor FP over R}.

\end{example}
We shall explain this observation as part of a general phenomenon in the following result.

\begin{theorem}\label{thm:doubling}
    Let $A_1,\ldots, A_r$ be graded AG $K$-algebras with $\reg(A_i)=d$ for all $1\leq i,j\leq r$. Suppose that for each $1\le i\le r$, $A_i$ is a doubling of some 1-dimensional Cohen-Macaulay algebra  $\tA_i$, then
the connected sum $A_1\#_K \cdots \#_K A_r$ is a doubling of $\tA_1\times_K \cdots \times_K \tA_r$.
\end{theorem}
\begin{proof}
    We proceed by induction on $r$. We first prove the base case where $r=2$. 
    Set $\tA_1 = R/\tilde{\fa}_1$ and $\tA_2 = S/ \tilde{\fa}_2$ and let $Q=R_1\otimes_K R_2$. By \cite[Lemma 1.5]{AAM} the ring $\tA_1 \times_K \tA_2$ is Cohen Macaulay of dimension one.
By \Cref{lem:facts}, our assumptions imply that for each $i$ we have exact sequences
\begin{equation}\label{doubles}
0 \to \omega_{{\tA}_i} (-d) \to \tA_i \to A_i \to 0.
\end{equation}
Considering these in degree zero we conclude that 
\begin{equation}
    \label{eq:can mod in low degree_multi factor}
[\omega_{{\tA}_i}]_{-d} = 0. 
\end{equation}

Combining the exact sequences \eqref{doubles} for $i\in\{1,2\}$ with the sequence in \eqref{exactFP}, we obtain the following commutative diagram of $Q$-modules with exact rows and middle column.
\begin{align}
     \label{commutative diagram_multi factor}
\minCDarrowwidth20pt
\begin{CD}
@. 0 @. 0 @.  \\
@.  @VVV  @VVV    \\
@.(\omega_{\tA_1} \oplus \omega_{\tA_2}) (- d) @>>{=}> (\omega_{\tA_1} \oplus \omega_{\tA_2}) (- d) @. \\
@. @VVV  @VVV    \\
0 @>>>  \tA_1 \times_K \tA_2  @>{\sigma}>>  \tA_1 \oplus \tA_2  @>{\mu}>>   K @>>>  0 \\
 @. @VVV   @VVV    @VV{=}V \\
0 @>>>  A_1 \times_K A_2  @>>>  A_1 \oplus A_2  @>>>   K @>>>  0 \\
@. @VVV  @VVV    \\
@. 0 @. 0 @.  \\
\end{CD}
\end{align} 
The vertical map $\tA_1 \times_K \tA_2 \to  A_1 \times_K A_2$ in \eqref{commutative diagram_multi factor} is uniquely determined by viewing $A_1\times_KA_2$ as a pullback in the category of $K$-algebras and utilizing the universal property of this categorical construction. Moreover, by the snake lemma, the kernel of this map is  the module $(\omega_{\tA_1} \oplus \omega_{\tA_2}) (- d) $.

Applying the functor $\Hom(-, Q)$ to  the diagram \eqref{commutative diagram_multi factor} yields a new commutative diagram \eqref{commutative diagram2_multi factor}. The middle row in \eqref{commutative diagram2_multi factor} comes from the top  of \eqref{commutative diagram_multi factor}, and the top row of \eqref{commutative diagram2_multi factor} contains the non-vanishing $\Ext$ modules for  the $Q$-modules in the middle row of \eqref{commutative diagram_multi factor}.  According to \Cref{rem: Thom class}, the map marked $\nu$ satisfies $\nu(1)=(\tau_{A_1},\tau_{A_2})$ after identifying $\omega_{A_1} \oplus \omega_{A_2}\cong A_1\oplus A_2$.

\begin{align}
     \label{commutative diagram2_multi factor}
\minCDarrowwidth20pt
\begin{CD}
@.  @. 0 @. 0 \\
@.  @. @VVV  @VVV    \\
0 @<<< K @<<< \omega_{\tA_1 \times_K \tA_2}  @<<<  \omega_{\tA_1} \oplus \omega_{\tA_2} @<<< 0 \\
@. @.  @VV\eta V  @VVV    \\
 @.  @. (\tA_1\oplus \tA_2)(d) @<<{=}< (\tA_1\oplus \tA_2)(d) @. \\
@. @.  @VVV  @VV\chi V    \\
 @.0 @<<< \omega_{A_1 \times_K A_2}  @<<< \omega_{A_1} \oplus \omega_{A_2} @<<< K @<<<  0 \\
 @.   @.  @VVV  @VVV \\
@.   @.0 @. 0 
\end{CD}
\end{align}

The snake lemma applied to \eqref{commutative diagram2_multi factor} yields a connecting isomorphism $\theta\colon K\to K$. 
Let $s\in \omega_{\tA_1 \times_K \tA_2}$ be such that $\xi(s)=\theta(1)$. Then $\chi(\eta(s))=\nu(1)$ can be identified with $(\tau_{A_1},\tau_{A_2})\in A_1\oplus A_2$, that is, $\eta(s)$ is equivalent to $(\tau_{A_1},\tau_{A_2})$ modulo the image of $\omega_{\tA_1}\oplus\omega_{\tA_2}$.

We want to compare the image of 
\[
\eta[-d]\colon \omega_{\tA_1\times_K\tA_2}(-d)\to \tA_1\oplus \tA_2
\]
and the kernel of the map $\mu$ from Diagram \eqref{commutative diagram_multi factor}. 
The image of $\eta[-d]$ is trivial in degree zero by Equation \eqref{eq:can mod in low degree_multi factor}. Since $K$ is concentrated in degree zero, the map $\mu$ has zero image in every degree other than zero. It follows that the image of $\eta[-d]$ is contained in $\ker \mu = \im \sigma \cong  \tA_1 \times_K \tA_2$. Hence $\eta[-d]$ induces an injective graded $Q$-module homomorphism 
\[
\delta \colon \omega_{\tA_1 \times_K \tA_2} (- d) \to \tA_1 \times_K \tA_2. 
\]
Its existence proves that  $\omega_{\tA_1 \times_K \tA_2} (-d)$ can be identified with an ideal of $\tA_1 \times_K \tA_2$.

The following diagram combines the left column of Diagram \eqref{commutative diagram_multi factor} and the top row of \eqref{commutative diagram2_multi factor}. By previous considerations indicating that $\delta(s)=\eta(s)$ is equivalent to $(\tau_{A_1},\tau_{A_2})$ modulo the image of $\omega_{\tA_1}\oplus\omega_{\tA_2}$, the diagram commutes provided that $\xi(s)$ is mapped by $\tau$ to $(\tau_{A_1}, \tau_{A_2})\in A_1\times_K A_2$. With this choice, the cokernel of $\tau$ is $A_1\#_K A_2$ by \Cref{Def_CS}.

\begin{align*}
     \label{commutative diagram3_multi factor}
\minCDarrowwidth20pt
\small{
\begin{CD}
@. 0 @. 0 @.  @. \\
@. @VVV  @VVV   @.    \\
@. (\omega_{\tA_1} \oplus \omega_{\tA_1})(-d) @>=>> (\omega_{\tA_1} \oplus \omega_{\tA_2})(-d) @.\\
@. @VVV   @VVV     \\
0 @>>>  \omega_{\tA_1 \times_K \tA_2}(-d)  @>>{\delta}>  \tA_1\times_K \tA_2 @>>>  C @>>>  0 \\
 @. @VV{\xi}V   @VVV   @. \\
0 @>>> K(-d)  @>>\tau >A_1\times_K A_2 @>>> A_1\#_K A_2 @>>> 0 \\
@.    @VVV @VVV    \\
@.  0 @. 0 \\
\end{CD}
}
\end{align*} 
Setting $C$ be the cokernel of $\delta$, the snake lemma yields an isomorphism $C\cong A_1\#_KA_2$. This shows that $A_1\#_K A_2$ is a doubling of $\tA_1\times_K \tA_2$, as desired for the base case of induction. 

Now, we assume that the AG $K$-algebra $A_1\#_K \cdots \#_K A_{r-1}$ is a doubling of $\tA_1\times_K\cdots \times_K\tA_{r-1}$. The base case applied to AG $K$-algebras $A_1\#_K \cdots \#_K A_{r-1}$ and $A_r$ implies by way of Remarks \ref{rem:multi-factor FP as iterated FP} and \ref{rem:multi-factor CS as iterated CS} that $A_1\#_K \cdots \#_K A_{r}$ is a doubling of $\tA_1\times_K\cdots \times_K\tA_r$ completing the proof.
\end{proof}

\Cref{thm:doubling} generalizes \cite[Theorem 5.5]{CLW}, which considered the case of AG algebras $A_1, \ldots, A_r$ of embedding dimension one, establishing an analogous doubling result.

\bigskip

\bibliographystyle{alpha} 
\bibliography{refs} 

\end{document}